\documentclass[vecarrow]{svmult}

\usepackage{makeidx}         
\usepackage{graphicx}        
\usepackage{psfrag}                             
\usepackage{multicol}        
\usepackage{amsmath}
\usepackage{amssymb}
\usepackage{amsfonts}
\usepackage{latexsym}
\usepackage[T1]{fontenc}
\usepackage{subfigure}                      
\usepackage{float}
\usepackage{mathrsfs}
\usepackage{mathptmx}         

\usepackage{url,amsbsy}
\usepackage{graphics,graphicx} 

\usepackage{longtable}

\usepackage[english]{babel}

\def\R{{\mathbb{R}}}

\def\V{{\mathbb{V}}}

\def\C{{\mathcal{C}}}

\def\N{\rm \hbox{I\kern-.2em\hbox{N}}}
\def\Z{\rm \hbox{Z\kern-.3em\hbox{Z}}}

\def\e{\varepsilon}
\def\d{\delta}

\def\O{\Omega}

\def\e{\varepsilon}

\def\XXint#1#2#3{{\setbox0=\hbox{$#1{#2#3}{\int}$} 
		\vcenter{\hbox{$#2#3$}}\kern-.5\wd0}}

\def\R{\mathbb{R}}

\def\O{{\Omega}}
\def\epsilon{\varepsilon}

\newcommand{\be}{\begin{equation}}
\newcommand{\ee}{\end{equation}}
\newcommand{\bes}{\begin{equation*}}
\newcommand{\ees}{\end{equation*}}
\newcommand{\bea}{\begin{eqnarray}}
\newcommand{\eea}{\end{eqnarray}}
\newcommand{\beas}{\begin{eqnarray*}}
	\newcommand{\eeas}{\end{eqnarray*}}

\newcommand{\F}{{\cal F}}

\begin{document}

\title*{Operator perturbation  approach for fourth order elliptic equations with variable coefficients }

\titlerunning{Operator perturbation  for fourth order elliptic equations with variable coefficients}

\author{J. \ Orlik, H. \ Andr\"a,  \& S. \ Staub}

\authorrunning{J. Orlik et al.}

\institute{J. \ Orlik,  H. \ Andr\"a and S. Staub \at Fraunhofer ITWM, Kaiserslautern, Germany; \\ 
	\email{julia.orlik@itwm.fraunhofer.de}, \email{heiko.andrae@itwm.fraunhofer.de}, \email{sarah.staub@itwm.fraunhofer.de}}
	
\maketitle
	

	
%
\abstract{
	The homogenization of elliptic divergence-type fourth-order operators with
	periodic coefficients is studied in a (periodic) domain. The aim is to find an operator with constant coefficients and represent the equation through a perturbation around this operator. The resolvent is found as $L^2 \to L^2$ operator using the Neumann series for the periodic fundamental solution of biharmonic operator. Results are based on some auxiliary Lemmas suggested by Bensoussan in 1986, Zhikov in 1991, Yu. Grabovsky and G. Milton in 1998, Pastukhova in 2016.  Operators of the type considered in the paper appear in the
	study of the elastic properties of thin plates. The choice of the operator with constant coefficients  is discussed separately and chosen in an optimal way w.r.t. the spectral radius and convergence of the Neumann series and uses the known bounds for ''homogenized'' coefficients. Similar ideas are usually applied for the construction of preconditioners for iterative solvers for finite dimensional problems resulting from discretized PDEs. The method presented is similar to Cholesky factorization transferred to elliptic operators (as in references mentioned above). Furthermore, the method can be applied to non-linear problems.}	

\setcounter{equation}{0}

	\bigskip
\def\e{\varepsilon}
\def\d{\partial}
\def\je{^{j}_{\e}}
\def\O{\Omega}
\def\vphi{\varphi}
\def\R{{\mathbb{R}}}
\def\Z{{\mathbb{Z}}}
\def\V{{\mathbb{V}}}
\def\C{{\mathbb{C}}}
\def\blangle{\Big\langle}
\def\brangle{\Big\rangle}
\def\F{{\mathscr{F}}}
\def\O{{\mathcal{O}}}
\def\fcpv{{m}}
\newcommand{\bA}{\ensuremath{\boldsymbol{A}}}
\newcommand{\bV}{\ensuremath{\boldsymbol{V}}}
\newcommand{\bP}{\ensuremath{\boldsymbol{P}}}
\newcommand{\bQ}{\ensuremath{\boldsymbol{Q}}}
\newcommand{\bsig}{\ensuremath{\boldsymbol{\sigma}}}
\newcommand{\bu}{\ensuremath{\boldsymbol{u}}}
\newcommand{\bfs}{\ensuremath{\boldsymbol{f}}}
\newcommand{\bx}{\ensuremath{\boldsymbol{x}}}
\newcommand{\bd}{\ensuremath{\boldsymbol{d}}}
\newcommand{\bT}{\ensuremath{\boldsymbol{T}}}
\newcommand{\bl}{\ensuremath{\boldsymbol{l}}}
\newcommand{\bI}{\ensuremath{\boldsymbol{I}}}
\newcommand{\bS}{\ensuremath{\boldsymbol{S}}}
\newcommand{\bC}{\ensuremath{\boldsymbol{C}}}
\newcommand{\bs}{\ensuremath{\boldsymbol{s}}}
\newcommand{\bw}{\ensuremath{\boldsymbol{w}}}
\newcommand{\bn}{\ensuremath{\boldsymbol{n}}}
\newcommand{\bg}{\ensuremath{\boldsymbol{g}}}
\newcommand{\eps}{\ensuremath{\varepsilon}}
\newcommand{\PoissonRatioSym}{\ensuremath{\nu}}
\newcommand{\poir}{\ensuremath{\PoissonRatioSym}}
\newcommand{\pxy}{\ensuremath{\PoissonRatioSym_{12}}}
\newcommand{\pxz}{\ensuremath{\PoissonRatioSym_{13}}}
\newcommand{\pyx}{\ensuremath{\PoissonRatioSym_{21}}}
\newcommand{\pyz}{\ensuremath{\PoissonRatioSym_{23}}}
\newcommand{\pzx}{\ensuremath{\PoissonRatioSym_{31}}}
\newcommand{\pzy}{\ensuremath{\PoissonRatioSym_{32}}}
\newcommand{\pij}{\ensuremath{\PoissonRatioSym_{ij}}}
\newcommand{\homog}{\ensuremath{\mathrm{hom}}}
\newcommand{\bnd}{\ensuremath{\partial}}
\newcommand{\jump}[1]{\ensuremath{\left[#1\right]}}
\newcommand{\RobNorm}{\ensuremath{r_n}}
\newcommand{\RobTang}{\ensuremath{r_t}}
\newcommand{\eprod}{\ensuremath{\otimes}}
\newcommand{\vrtr}{\ensuremath{\delta_{\bg}}}
\newcommand{\ggrad}{\ensuremath{\nabla_{\bg}}}
\newcommand{\grag}{\ensuremath{\delta_{\bg}}}
\newcommand{\vrtw}{\ensuremath{\delta_{\bw}}}
\newcommand{\Ahom}{\ensuremath{{\bA^{\homog}}}}
\newcommand{\Ahomc}{\ensuremath{{A^{\homog}}}}
\newcommand{\sAhom}{\ensuremath{\bA^{\homog}_{\mathrm{sym}}}}
\newcommand{\vAhom}{\ensuremath{\vrtr\Ahom}}
\newcommand{\vuh}{\ensuremath{\vrtr\bu^0}}
\newcommand{\gadm}{\ensuremath{U_{\bg}}}
\newcommand{\strcom}{\ensuremath{\sigma}}

\section{Periodic boundary value problem problem}

Let $Y' \in \R^d$ be a bounded domain with a heterogeneous structure, such that $Y'$ is a representative periodic cell of the structure. For the periodic functions $\varphi$ on $Y'$ with a zero mean value 
$\langle \varphi \rangle = 0$ we introduce the Sobolev space
\[
H^2_{per[0]}(Y') = \{\varphi \in H^2_{per}(Y') : \langle \varphi \rangle = 0\}.
\]
Furthermore, we use the abstract index notation for tensors in $\R^d$. We introduce latin indices $ijkl=1,...,d$
and and unit vectors $e_i$, $i=1,...,d$ . 
We consider the periodic boundary value problem (PBVP) to find $w\in H^2_{per[0]}(Y')$ from
\begin{align}
&\frac{\partial^2 }{\partial y_k \partial y_l}\left(C_{ijkl}(y) 
\left(E_{0\  i j}+\frac{\partial^2 w(y)}{\partial y_i \partial y_j}\right) \right) =0, \quad \forall y \in Y'.
\label{4-plane}
\end{align}
Here, the tensor $E_0\in \R^{d\times d}$ denotes the given mean value of rotational deformations.\\
The coefficients of the fourth-order tensor $C\in L^\infty(Y',\R^{d\times d\times d \times d})$ are bounded and satisfy the symmetry conditions 
\begin{align}
C_{ijkl} = C_{klij},\  C_{ijkl} = C_{jikl} = C_{ijlk},\;
\exists \alpha > 0 : \ \alpha \xi : \xi \le  \xi : C(y) : \xi \le \alpha^{-1} \xi : \xi
\label{coefbd}
\end{align}
for any symmetric non-zero tensor $\xi \in \R^{d\times d}$ and $y \in Y'$.

\begin{remark}
In the two-dimensional case $Y' \subset \R^2$, one important application is the bending of plates made of heterogeneous materials. 
According to \cite{fjm}, for thin three-dimensional plate-like structures, where the ratio between the thickness and the representative lateral size tends to zero, the theoretical limits leads to averaging in the thickness and then solving the plate bending equation with variable coefficients in 2D-domains. 
\end{remark}

\subsection{Problem in the weak and operator form}
For a scalar function $w$ and a tensor-valued function $g$ we introduce the following operators to shorten the notation:
\begin{eqnarray}\label{2.2}
&&D w (y):=  \nabla \otimes \nabla w(y) = \frac{\partial w}{\partial y_{k}\partial y_{l}}(y) e_k \otimes e_l ,\\
&& D^*g(y) := \nabla \cdot \nabla \cdot g(y) =\mathrm{div}\  \mathrm{div} \, g(y) = \frac{\partial^2 g_{kl}}{\partial y_{k}\partial y_{l}}(y) ,\quad \Delta^2=D^*D,\\
&& [Aw](y):=D^*(C(y)Dw),
\end{eqnarray} 
where Einstein's summation convention is used on the last line.
So equation \eqref{4-plane} can be written in the short form
\begin{eqnarray}
	&& D^*(C(y) (E_0 + Dw(y)))=0, \ \  y\in Y', \ \  \textrm{or} \quad Aw=-D^*g, \quad g:=C(y): E_0. 
	\label{strong}
\end{eqnarray}
%
The weak formulation of the problem \eqref{4-plane} is given by
\begin{equation}
\langle C(D w +E_0):  D\varphi \rangle = 0,\  \forall \varphi \in  
H^2_{per[0]}(Y').
\label{weak}
\end{equation}
where $C=C(y)$
%
Now, we introduce the bilinear form 
	\begin{align}
	c(w, v):=&\int_{Y'}    
	C_{ijkl}(y) [\nabla \otimes \nabla w]_{i j}  [\nabla \otimes \nabla v]_{kl}(y)  dy ,\quad \forall w,v\in H^2_{per[0]}(Y'),
	\end{align}
which is symmetric, because the tensor of coefficients is symmetric,
	continuous and coercive. \\
	We also introduce an operator $L: H^2(Y')\to H^{-2}(Y')$ and the right-hand side  functional 
	\begin{equation}
	L 
	w:= c(w,\cdot) \label{Cx},\quad {l_y}(v)=\int_{Y'}D^*: gv(y)dy.
	\end{equation}
	Our problem \eqref{4-plane} can be reformulated in the operator form: Find $u\in H^2_{per[0]}(Y')$ for ${l_y}\in H^{-2}(Y')$ and $L\in {\cal L}(H_{per}^2(Y'), H^{-2})$, satisfying
	\begin{equation}
	L 
	w={l_y}.
	\label{operat-eq}
	\end{equation}
According to \cite{Pastukhova}, the norm in
$H^2_{per[0]}(Y')$ can be introduced in one of the  equivalent
ways:
\begin{eqnarray}
	||	\varphi ||^2
	_{H^2_{per[0]}(Y)}
	= <|\nabla^2\varphi|^2 + |\nabla \varphi|^2 + | \varphi|^2>, \quad
	||	\varphi ||^2_{H^2_{per[0]}(Y)}
	= <|\nabla^2\varphi|^2>
	\label{poinc}
\end{eqnarray}
The equivalence of the norms is ensured by the Poincare inequality
$$<|\varphi|^2>\  \le\  C_P <|\nabla \varphi|^2> \ \forall \varphi \in  H^1(Y')\  \textrm{with}\  <\varphi> = 0.$$\\
%
The proof of the following existence and uniqueness result for the week solution to problem \eqref{operat-eq} can be found, e.g. in \cite{Pastukhova}:
\begin{lemma}
	By the Riesz and Lax-Milgram theorem, the bending problem admits a unique solution in $H^2_{per[0]}(Y')$.
	$w$ depends continuously on $E_0$, since \\ $L\in {\cal L}(H^2_{per[0]}(Y'), H^{-2}(Y'))$ with continuous bounded inverse  $$||L^{-1}||_{{\cal L}(H^{-2}(Y'),H^2_{per[0]}(Y'))}\le \frac{1}{\alpha}.$$
	\label{inv}
\end{lemma}
\subsection{Integral approach for constant coefficients. Green's identity and function }
Assume for this subsection that problem \eqref{strong} has constant coefficients.
For any periodic
functions $w,v \in C^4(\bar{Y'})$
 we have the second Green identity,
\be
\label{2.5-2}
\int\limits_{Y'}\big[\,  v\,Aw - Av\,\,w\,\big]\,dx=
\left\langle  T^{+}w\,,\, \gamma^+v\right\rangle _{_{\partial Y'}}
- \left\langle  T^{+}v\,,\, \gamma^+w\right\rangle _{_{\partial Y'}}.
\ee
Take   $v =G$,  
where 
$G\star \equiv (L)^{-1}
$ and
$G(y,x)$ is the Green's function for the periodic 
problem $Av=\delta(x)$ and the second expression gives $w$.  The first and second boundary intergrals disappear, since the Green-function, $DG$, $w$ and $Dw$ satisfy  periodic  conditions.
\be
\label{2.5-2}
-w(x)+ \int\limits_{Y'}\,  G(x,y) \,Aw(y) \,dy=
0, \quad x\in Y'.
\ee
Replacing $Aw$ by the right-hand side of equation  $\eqref{strong}_2$,
we obtain
$$
w(x)=\left(G\star D^*g\right)(x)=\int_{Y'}G(x,y)D^*g(y)dy.
$$ 
Note that 
$$
Dw(x)=D\left(G\star D^*g\right)(x)=\int_{Y'}D_xG(x,y)\cdot D_x^*g(y)\,dy=\int_{Y'}D_y (D_yG(x,y)):g(y)\,dy.
$$
	\begin{lemma}
		Let $C_0 \in \R^{d\times d \times d \times d}$	be a constant symmetric tensor and $w\in H^2(Y')$ a scalar function, then $D^*(C_0w)= C_0: Dw.$
		\label{C0}
	\end{lemma}
	
	\begin{proof}
		Here we  use the following rule for a constant matrix $C_0$ and a scalar $w$: 
		$$
		D^*(C_0w)=wD^*C_0 + C_0: Dw+ 2(\mathrm{div} \, C_0)\cdot \nabla w; 
		$$
	\end{proof}
	\begin{equation}
	D^*(C_0:D w)=
	D^*D^*(C_0 w)= C_0::(DD w)= -D^*(g). 
	\label{511}
	\end{equation}

	\begin{remark}
		Note that a fundamental solution for an anisotropic  plate  exists and is known. E.g. a periodic fundamental solution can be found in a series. Anyway, it is possible to write a Newton-type potential, $G$, which will be inverse to the operator $C_0::DD[ \cdot]$.
	\end{remark}
%
%
%
	Let us check, that $G\star: H^{-2}(Y')\to H^2_{per[0]}(Y')$, i.e., that $\delta \in H^{-2}(\R^2)$.
	\begin{definition}
		We define the norm $\|u\|_{H^s(\R^d)}^2$ for $u\in H^s(\R^d)$ as
		\begin{align}
		\|u\|_{H^s(\R^d)}^2 = \int_{\R^d} (|\xi|^2 + 1)^s |\tilde{u}|^2 d\xi,
		\end{align}
		where $\tilde{u}$ is the Fourier transform.
	\end{definition}
	This definition can also be found on \cite[p. 76]{McLean}.
	\begin{lemma}\label{lemma_coroll}
		The Dirac function $\delta$ belongs to $H^s(\R^d)$ for any $s<\frac{-d}{2}$. 
	\end{lemma}
	\begin{proof}
	From the transformation formula for integrals it follows
		\begin{align}
		&\|\delta\|_{H^s(\R^d)}^2 = \int_{\R^d}\frac{(1 + |\xi|^2)^s}{(2\pi)^d} d\xi \sim \frac{1}{(2\pi)^n} \int_1^\infty \rho^{2s} \rho^{n-1} d\rho \nonumber \\
		&= \frac{1}{(2\pi)^n(2s + n)}  < \infty,\quad \textrm{if}\quad s<\frac{-d}{2}.
		\end{align}
	\end{proof}
	
	\begin{corollary}
		$\delta(x)\in H^{-2}(\R^2)$.
	\end{corollary}

\begin{remark}
	Following the first Green's Identity
	$$
	\left\langle  T^{+}w\,,\, \gamma^+v\right\rangle _{_{\partial Y'}}=
	\int\limits_{Y'}\big[\,  v\,Aw + c(w,v)\,\big]\,dy
	$$ and accounting for the fact, that $
	\left\langle  T^{+}w\,,\, \gamma^+v\right\rangle _{_{\partial Y'}}$ vanishes because of the periodicity, we can conclude that $vAw=-Lwv$.
	Because of Lemma \ref{inv}, $G\star \equiv (L)^{-1}$ exists for variable coefficients as well. We need to find a way to construct this Newton-type potential.
	\end{remark}

\subsection*{Potential polarization field}
	We refer to \cite{EM} and repeat the same approach applied to the second order elliptic PDE there to the higher order elliptic PDE with a symmetric coefficient matrix.\\
	The idea is to modify the problem, given in the variational formulation to the integral equation of the second kind, with the integral operator mapping in the same space.
	To solve the desired equations, we consider a related problem with a reference operator $L_0$, corresponding to the constant reference tensor $C_0$.
	\begin{equation}
	D^*(C_0:E(y))+D^*((C(y)-C_0):E(y))=0.
	\label{36}
	\end{equation}
Denoting 
$C(y)-C_0=\delta C(y)$ yields
	\begin{equation}
	-D^*(C_0:(E(y)-E_0)=D^*(\delta C(y):E(y)), \quad \textrm{or} 
	\label{366}
	\end{equation}
		\begin{equation}
		D^*(C_0:Dw(y))+D^*(\delta C(y):Dw(y))=-D^*(C_0:E_0).
			\label{367}
		\end{equation}
	If the polarization tensor $P=\delta C(y):E(y)$ were known,
	\begin{equation}
	D^*[C_0:Dw]= -D^*P,
	\label{5}
	\end{equation}
	the related problem could be solved using the periodic Green's operator. 
	\begin{lemma}
		Problem \eqref{5}
		has a periodic unique solution $w\in H^2_{per[0]}(Y')$,  
		which can be represented  by \eqref{5} through an integral operator
		$\Gamma \star: L^2(Y',\R^{d\times d})\to L^2(Y',\R^{d\times d})$ given by  the formula
		\begin{equation}
		[\Gamma\star P](y)\equiv -DL_0^{-1}D^*P(y), \quad \forall P\in L^2_{per}(Y',\R^{d\times d})
		\label{Pe}
		\end{equation}
		$$
		[\Gamma\star P](y)=\int_{Y'}D_y [D_xG(x,y)]:P(y)dy,
		$$ 
		$G(x,y)$ 
		is the Green's function for the periodic 4th-order problem 
		$$
		C_0::(D_yD_y G)=\delta(y-x),
		$$
		where $\delta(y-x)$ is the Dirac delta distribution.
		\label{lem31}
	\end{lemma}
	Now we can rewrite \eqref{367} as the following
\begin{equation}
E(y)-E_0=-\Gamma\star P(E), \quad \textrm{or} \quad E(y)+\Gamma\star \delta C(y):E(y)=E_0
\label{4}
\end{equation}
and
\begin{equation}
E(y) =(I-\Gamma\star \delta C(y))^{-1}E_0=:R\star E_0. 
\label{45}
\end{equation}
with 
\begin{equation}
\Gamma\star \delta C:=-DL_0^{-1}D^*\delta C, \quad [\Gamma \delta C](x,y):=D_y [D_xG(x,y)]\delta C(y)\in L^1(Y'\times Y')
\label{ink}
\end{equation}

\subsection{Orthogonal decomposition of the 4th-order differential operator on $Ker$ and $Im$ }

	Our problem reduces to \eqref{5} and  finding
	$$ P\equiv (C(y)-C_0):E(y) \in V_{pot[0]}(Y',\R^{d\times d}).$$
	\begin{remark} If $C_0$ would be the effective coefficients (see (3.11) in \cite{Pastukhova}), then
	$$<C(y):E(y)>=C_0:<E>.$$	
	\end{remark}
	Following ideas from \cite{Grab99}, \cite{Pastukhova}, \cite{Jikov}, we decompose the $L^2$-space into the solinoidal and potential in the sense of higher-order $\mathrm{div}^n$-free and $\mathrm{curl}^n$-free, or higher-order potential tensor spaces:\\
	$ V_{sol}(Y',\R^{d\times d})=\{P = \{g_{sh}\}_{s,h}\in L^2(Y',\R^{d\times d})\ \textrm{ be symmetric,}\ 
	 \textrm{s. t.}\quad
		D^*P = 0.\}$\\
			$ V_{pot}(Y',\R^{d\times d})= \left\{D\Gamma,\  \Gamma \in H^2_{per}(Y') \right\}. 
			$\\
According to \eqref{weak}, one has the orthogonality property:\\
$
V_{sol}(Y',\R^{d\times d} )=V_{pot}(Y',\R^{d\times d} )^\perp$\\
$
\quad =\{P \in  L^2(Y',\R^{d\times d} ):\ <P: e> = 0\quad  \forall e \in V_{pot}(Y',\R^{d\times d} )\}.
$\\
Define another potentilal field	\\					
		$ V_{pot[0]}^*(Y',\R^{d\times d})= \left\{P \in V_{sol}(Y',\R^{d\times d}) \  \textrm{ and} 
		\   \left<P\right> = 0.\ \ \exists\  \Gamma\in H^2_{per}(Y',\R^{d\times d\times d\times d}),\right.
		$\\ 
		$\left. \ \textrm{s.t.}\ D^*\Gamma = P, \right.
		\left. 
		\Gamma^{sh}
		_{ij}= \Gamma^{sh}
		_{ji} \quad \textrm{(sym)},\quad
		\Gamma^{sh}_{ij} = -\Gamma_{sh}^{ij} \quad \textrm{(skew-sym)}\right\}.
	$\\
$
V_{sol[0]}(Y',\R^{d\times d} )
=\{P \in  L^2(Y',\R^{d\times d} ):\ D^*P= 0,\quad  <P>=0\}.
$\\
$
V_{sol}(Y',\R^{d\times d} )=V_{sol[0]}(Y',\R^{d\times d} )\otimes \R^{d\times d}
$\\
$
V_{pot}(Y',\R^{d\times d} )=V_{pot[0]}(Y',\R^{d\times d} )\otimes \R^{d\times d}
$
\begin{lemma}
	For $g \in V_{sol,[0]}(Y')$   $\left<f,g\right> \equiv \int_{Y'} 
	f: g\, dy= 0$, $ \forall f \in V^*_{pot}(Y')$. 
\end{lemma}

\begin{proof}
	Let $\Gamma$
	be a solution of the periodic boundary value problems for the
	equations\\
	$D^*\Gamma^{ij}_{sh} =  g_{sh}$,  
	$D^*  g=0$.\\  
	Then
	\begin{eqnarray*}
		\int_{Y'} D_{ij}^*\Gamma^{ij}_{sh}  
		f_{sh} \,dy 
=\int_{Y'} \Gamma^{ij}_{sh}   D_{ij} f_{sh}\,dy \stackrel{\Gamma^{ij}_{sh} \textrm{skew sym.}}{=}0. 
	\end{eqnarray*}
\end{proof}

I.e., the following Weyl's decomposition (see \cite{Jikov}), known for the theory of second order elliptic operators, should be valid also for the higher-order operators:  
\begin{align*}
L^2(Y')&=V_{sol[0]}(Y',\R^{d\times d} )\otimes V_{pot[0]}(Y',\R^{d\times d} )\otimes \R^{d\times d} \\
&=V_{pot[0]}(Y',\R^{d\times d} )\otimes V_{sol} (Y',\R^{d\times d} ).
\end{align*}
%
%
%
%
%

We describe the way to construct a Newton potential. Let us
	take the Fourier series expansion 
	$$
	g_{ij}(y)= \sum_{0\neq n \in \Z^d}
	g^{ij}_n e^{2\pi n\cdot y\sqrt{-1}}.
	$$
	Then the condition $D^*P = 0$ means that 
	\begin{equation}
		g^{ij}_n n_i n_j = 0\quad  \textrm{for all} \quad n.
		\label{316}
	\end{equation}
	We introduce the matrix $\Gamma^{sh} = \{\Gamma^{sh}_{ ij} \}_{i,j}$ with the Fourier series coefficients 
	$$
	\Gamma^{sh}_{ ij,\ n} =(-g^{ij}_n n_s n_h + g^{sh}_n n_i n_j)|n|^{-4}(-4\pi^2)^{-1}.
	$$ 

	%
	%
	%
	
	The $n$-th term of the Fourier series satisfies 
	\begin{eqnarray}
		&& D^*[\Gamma^{sh}_n e^{2\pi n\cdot y\sqrt{-1}}]= \frac{\partial^2}{\partial y_i \partial y_j}
		\Gamma^{sh}_{ij,n}e^{2\pi n\cdot y{\sqrt{-1}}} \nonumber \\
		&&	= e^{2\pi n\cdot y\sqrt{-1}}|n|^{-4}(-g^{ij}_n n_s n_h n_i n_j + g^{sh}_n n_i n_j n_i n_j)=e^{2\pi n\cdot y\sqrt{-1}}g^{sh}_n.
	\end{eqnarray}
	Here relation \eqref{316} and the identity $n_i  n_j n_i n_j = |n|^4$ were taken into account. 

\subsection{Periodic fundamental solution of the biharmonic equation}

Assume now that the tensor $C_0$ acts on a matrix $\xi$ as $C_0\xi  = \lambda_0 C_0(\textrm{Tr} \, \xi)I$, where $\textrm{Tr} \, \xi = \xi_{ii}$ is the
trace of the matrix $\xi$, $I$ is the identity matrix, and $\lambda_0$ is a scalar. Obviously, $C_0\xi\cdot \xi =
\lambda_0(\textrm{Tr} \, \xi)^2$, and the matrix $\xi = D\varphi$ satisfies $\textrm{Tr} \, D\varphi = \Delta \varphi$, $C_0D\varphi \cdot D\varphi = \lambda_0 \Delta \varphi \Delta \varphi$. \\
Let $E=E_0+Dw$ be the bending deformations of the plate, or the curvature.
$$
E_{n+1}(y) =\Gamma \star [(I-\lambda^{-1}_{0}C(y)):E_n(y)],\quad n=0,1,2,3,...,
$$
where
$$
[\Gamma\star E](x)=\int_{Y'}D_y (D_x G(x,y)):E(y)dy,
$$ 
$G(x,y)$ 
is the Green's function for the periodic biharmonic problem, i.e. for each fixed $x\in Y'$ the function $G(x,y)$ satisfies the equation 
\begin{equation}
\Delta^2_y G=\delta(y-x),
\label{biharm}
\end{equation}
where $\delta(y-x)$ denotes the Dirac measure.\\ 
We redefine here $g=-\lambda_0^{-1}g$. Let us denote by $(\Delta^2)^{-1}D^*g$ for $g\in (L^2(Y'))^{d\times d}$ a solution of the periodic problem 
\begin{equation}
\Delta^2 w=-D^*g,\quad v\in H_{per[0]}^2(Y'),
\label{Delta}
\end{equation}
which exists due to Lemma \ref{inv}.\\
%
We can construct a periodic fundamental solution in the following way.
We	take the Fourier series expansion 
	$
	G(y)= \sum_{n \in \Z^d}
	c_n(G) e^{2\pi n\cdot y\sqrt{-1}},
	$
	where the Fourier coefficients
	$$
	c_n(G)=\int_{Y'}G(x)e^{2\pi n\cdot y\sqrt{-1}}dy, \ \ n \in \Z^d.
	$$
	The Fourier series for the Dirac sequence has coefficients $c_n(\delta)=1$:
	$
	\delta(y)=\sum_{0\neq n \in \Z^d}
	e^{2\pi n\cdot y\sqrt{-1}}.
	$
	The $n$-th term of the Fourier series satisfies 
	\begin{eqnarray}
	&&\Delta^2[c_n(G) e^{2\pi n\cdot y\sqrt{-1}}]\equiv D^*D[c_n(G) e^{2\pi n\cdot y\sqrt{-1}}]
	\nonumber \\
	&&
	= 
	\sum_{i=1}^d \sum_{j=1}^d\frac{\partial^2}{\partial y_i \partial y_j}
	\left(	\frac{\partial^2}{\partial y_i \partial y_j}
	c_n(G)e^{2\pi n\cdot y{\sqrt{-1}}}\vec{e}_i\otimes \vec{e}_j\right) \nonumber \\
	&&	= - e^{2\pi n\cdot y\sqrt{-1}}(2\pi)^{4}c_n(G) 
	(n_i n_j n_i n_j) .
	\end{eqnarray}
	Here the identity $n_i  n_j n_i n_j = |n|^4$ is taken into account. 
	$$
	c_n(G)=-(2\pi)^{-4}|n|^{-4}
	$$
	\begin{equation}
	G(y)=-(2\pi)^{-4}\sum_{0\neq n \in \Z^d}
	|n|^{-4} e^{2\pi n\cdot y\sqrt{-1}}.
	\label{perGreen}
	\end{equation}
\begin{remark}
According to \cite{Willis} (see Appendix), if take $Y'$ such that it will contain many periodicity cells,
it is possible to take the fundamental solution and do not look for the periodic Green's function. Since we require that the right-hand side of the equation has a zero mean-value over $Y'$,  the boundary values oscillate around zero and, according to the St.Veinant hypothesis, the corresponding boundary layer decays exponentially. Its thickness is proportional to the period of the structure (or characteristic size of inclusions related to the distance between them) (see, e.g., \cite{Toupin}).
\end{remark}

%
%

\section{Neumann series and its convergence estimate by spectral properties}

	\begin{definition}
		Let ${\cal A}$ be a Banach-algebra with unit, $\Gamma\in {\cal A}$ and $\sigma(\Gamma\star)$ be the spectrum of $\Gamma\star$.
		$\rho(\Gamma\star )=\sup\{|\lambda| \hspace{2mm} | \lambda \in \sigma(\Gamma\star)\}$ 
		is the  spectral radius of $\Gamma\star $.
		\label{defspecrad} 
	\end{definition}

	The class of kernels of integral 
	operators  on the compact set 
	forms a Banach algebra with product $\star$ without a unit. 
	We can add  the identitity as a unit to our algebra ${\cal A}$.

	\begin{definition}
		Let $Y'\subset \R^d$ be a compact set, and $\Gamma\star\delta C \in {\cal L}(L^2(Y'))$. 
		Then 
		there exists a resolvent $\underline{\tilde{R}}(\lambda)
		$ of $\Gamma\star$
		in ${\cal L}(L^2(Y')))$, given by 
		\begin{equation}
		\underline{\tilde{R}}(\lambda):=(\lambda I -\Gamma\star \delta C)^{-1}=\sum_{j=0}^\infty \frac{(\Gamma\star\delta C)^j}{\lambda^{j+1}}, \quad \forall \lambda \in \C\setminus\sigma(\Gamma\star \delta C).
		\label{Neum}
		\end{equation} 
		\label{specrad}
	\end{definition}
	%
%
%
%
%
%
%
%
%
%
In order for solution to \eqref{4} to exist, i.e., 
		\eqref{Neum}
		the Neumann series to converge for $\lambda =1$, the spectral radius should be $\rho(\Gamma \star \delta C)<1$.	And this depends on the choice of the constant coefficient matrix $C_0$. We postpone this question to the next section and discuss  properties of the resolvent kernel now. 
	We recall further properties of the resolvent and refer to \cite[Corol. 5.5.1.]{pseu}. 
	\begin{lemma}
		Let $R\star$ be an integral operator and $R$ be its kernel and $c_n(R)$ ($n\in \Z^d$) its Fourier coefficients.  If $|c_n(R)|\le c|n|^\alpha$  with some $\alpha\in \R$, then $R\star \in
		{\cal L}(H^{\lambda}, H^{\lambda-\alpha})$ for any $\lambda \in \R$. Moreover, if $c_1|n|^\alpha \le |c_n(R)| \le c_2|n|^\alpha$ ($n \in \Z^d$) 
		with some positive constants $c_1$ and $c_2$, then $R\star$ builds an isomorphism between
		$H^{\lambda}$ and $H^{\lambda-\alpha}$ for any $\lambda \in \R$.
	\end{lemma}
	
	\begin{lemma}
		Let  $R$ be a resolvent kernel to $G$ and $c_n(R)$ ($n\in \Z^d$) its  Fourier coefficients.  $|c_n(R)|\le c|n|^{-4}$  and $R \in L^1(Y',
		{\cal L}(H^{-2}, H^{2}))$. That means, the resolvent kernel is weakly singular and has the same leading singularity as $G$ (the fundamental solution of the 4th order PDE with  constant coefficients).  The last statement is true for the kernels' pare $\tilde{R}$ and $\Gamma$ also.
	\end{lemma}
	\begin{proof}
		Owing to \cite[Definition 9.3.1]{Gri}, 
		the resolvent kernel $R$ of kernel $G$ satisfies the following equality: 
		\begin{equation}
		R+G\star R=R+R\star G=G.
		\label{eqres}
		\end{equation}
		In order to check that $|c_n(R)|\le c|n|^{-4}$,	we again apply the Fourier series technique, replacing the Fourier transform for the peiodic functions. According to \cite{pseu}, the coefficients of the series, corresponding to a product of two functions correspond to the product of their coefficients.  Then
		$
		(I+c_n(G))c(R)=c_n(G)
		$, equivalently (see \cite[Th.2.8]{Gri}) $(I+c_n(G))(I-c(R))=I$. Recall that
		$c_n(G) =-(2\pi)^{-4}|n|^{-4}.$ Hence, we can see by the developing in the series
		$$
		|c(R)|= \sum_{l=1}^{\infty}((2\pi)^{-4}|n|^{-4})^l .
		$$
		Furthermore, the expression \eqref{Neum} of the resolvent by the multiple multiplication of $G\star$  
		maps into the same algebra and the same space. 
	\end{proof}	
\begin{theorem}
If for $\Gamma \star \delta C$ defined by \eqref{ink}
$\rho( \Gamma \star \delta C)<1$, 	
the unique solution of problem \eqref{367}, $w\in H^2_{per[0]}(Y')$, exists
%
%
and satisfies the estimate
\begin{eqnarray}
\| w\|_{H^2(Y')} \le 
\sup_{j\in \N}\left\|\left(D_y D_xG \star\right)^j\right\|_{L^1(Y'\times Y')} \sum_{j=0}^\infty \left\|\frac{\delta C}
{\alpha(C_0)}\right\|^j_{(L^\infty(Y'))^{d\times d}}
{\|E_0\|_{L^2(Y')}
}
,
\label{kbd}
\end{eqnarray}
where $G$ is the periodic fundamental solution of the $4th$-order PDE with constant coeficients $C_0$.
\end{theorem}
\begin{proof}
	This estimate is based  
	on \eqref{poinc}, estimate of the resolvent, given by the Neumann series \eqref{Neum} and 
	the Lax-Milgram theorem. 
\end{proof}
	In recent works \cite{Suslina}, \cite{Niu}, \cite{Shen},  also the leading term asymptotics for the resolvent were discussed. We note, that our results and techniques can also be generalized for high-order PDEs.

\section{Bounds on $C_0(x)$}
	
	As we mentioned above, the Neumann series for problem \eqref{45} converges for $\lambda =1$, if the spectral radius  $\rho( \Gamma \star \delta C)<1$.	And this depends on the choice of the constant coefficient matrix $C_0$.\\	
	Let $\lambda$ and $E_\lambda$, be
	an eigenvalue and a corresponding eigenvector of 
	\begin{eqnarray*}
	\lambda E_\lambda =- 
	\Gamma\star  ((C(x)-C_0) : E_\lambda), \ \  
	(\lambda C_0+(C(x)-C_0)) : E_\lambda =0, \ \  
	C'(x) : E_\lambda =0,
\end{eqnarray*}
	where $C'(x)=C(x)-(1-\lambda)C_0$.
	We need, that for $|\lambda|\ge 1$ the last problem has only trivial solution, i.e.
	$$
	\xi : C' : \xi>0,\quad \textrm{or}\quad \xi :  C' : \xi<0,\quad \forall \xi \in \R^{d\times d}
	$$
%
%
%
%
%
	In order for $\lambda$ to be a non-trivial eigenvalue,
	$C'$ cannot be either positive definite or negative definite. 
	One must have, for the eigenvalues of the matrix
	$$\exists x_1;\  \mu'_i(x_1)< 0, \quad \textrm{and} \quad \exists x_2; \ \mu'_i(x_2)> 0, \ \forall i=1,...,d.
	$$
	In other words:
	\begin{eqnarray}
	\min_i\left\{\min_x\left(1-\frac{\mu_i(x)}{\mu_{i0}}\right)\right\} \le
	\lambda \le \max_i\left\{\max_x\left(1-\frac{\mu_i(x)}{\mu_{i0}}\right)\right\}. 
	\label{21}
	\end{eqnarray}

	A sufficient condition for the scheme to converge is $\rho<1$, i.e., $|\lambda|<1$, hence
	$\mu_{i0}>
	\mu_{i}(x)/2>0$, $i=1,...d$.
	Under this constraint, we look for $\mu_{i0}$ giving bounds in \eqref{21} with opposite sign
	and equal absolute value. This is ensured by
	\begin{equation}
	\mu_{i0} =\frac{1}{2}(\min_x \mu_{i}(x) + \max_x \mu_{i}(x)).
	\end{equation}
	With these parameters, the spectral radius of $\Gamma \star \delta C$ is bounded from above by
	\begin{eqnarray}
	\rho\le \max_i\left\{\frac{\max_x \mu_{i}(x)-\min_x \mu_{i}(x)}{\max_x \mu_{i}(x)+\min_x \mu_{i}(x)}\right\}. 
	\label{22}
	\end{eqnarray}
	
	Furthermore, owing to \cite{Moulinec2}, for the high contrast coefficients, is better to choose
	\begin{equation}
	\mu_{i0} =-\sqrt{\min_x
		(\mu_{i}(x))\cdot  \max_x
		(\mu_{i}(x))}.
	\end{equation}

\subsection{Voigt-Reuss bounds for the effective coefficients}\label{homest}
Let $L^{hom}$ be the effective operator with tensor  $C^{hom}$, defined (see (3.6) in \cite{Pastukhova}) by 
\begin{equation}
J_0 =C^{hom}E_0 =<C(x)E(x)>= <C(y)(D w (y) + E_{0})>,
\label{32}
\end{equation}
where $J_0 =<J>$ the mean value of the field $J$. The approach to find the effective tensor for the fourth order PDE is presented in \cite{Pastukhova}. The tensor $C_{hom}$ inherits the properties of symmetry and positive definiteness.\\
The following Voigt-Reuss estimates for the homogenized tensor are known, e.g. form\cite{Jikov} for the homogenization of the second-order elliptic operators and were generalized in \cite{Pastukhova} for higer order PDEs:
\begin{align}
\left<C^{-1}(x)\right>^{-1} \le C_{hom} \le\  \left<C(x)\right>,
\label{VR}
\end{align}
called the Voigt-Reuss bracketing. Here $C^{-1}$ is the inverse tensor for $C$ and $<\cdot>$ is an averaging over the periodicity cell. The following examples cane be found, e.g., in \cite{Jikov}
\begin{example}
	(i)	For a laminate with two isotropic faces along the faces with conductivities $\alpha$ and $\beta$ the effective homogenized coefficient is the mean value of both and across the layers the harmonic mean.
	\\
	(ii) For the chess-board structure, assuming $C_1=\alpha I$, $C_2=\beta I$, $C_{hom}=\sqrt{\alpha \beta}I$ and for a monocrystal $C_1$ with eigenvalues $\lambda_1,\lambda_2$, and $C_2= R^T C_1 R$,
	$R$ is an orthogonal matrix. 
	The homogenized tensor will be $C_{hom}=\sqrt{\lambda_1 \lambda_2}I$
\end{example}

%


\end{document}